\documentclass[leqno,11pt]{amsart}
\usepackage{amsmath,amstext,amssymb,amsopn,amsthm,mathrsfs,wasysym,dsfont,comment}

\theoremstyle{plain}

\allowdisplaybreaks

\newtheorem{thm}{Theorem}[section]
\newtheorem{cor}[thm]{Corollary}

\newtheorem{ob}[thm]{Observation}

\newtheorem{lem}[thm]{Lemma}
\newtheorem{rem}[thm]{Remark}
\newtheorem*{ex*}{Example}

\newcommand{\N}{\mathbb{N}}

\newcommand{\R}{\mathbb{R}}

\def\1{\boldsymbol{1}}

\def\8{\infty}

\def\8{\infty}

\def\a{\alpha}
\def\b{\beta}

\def\t{\theta}
\def\vp{\varphi}
\def\g{\gamma}

\def\eL{L_{\mu}}
\def\W{dW_\mu}
\def\pimu{d\Pi_{\mu - 1/2}(u)}
\def\piku{d\Pi_{\kappa - \mathbf{1}/2}(u)}
\def\distB{d_{B}}

\DeclareMathOperator{\dist}{dist}
\DeclareMathOperator{\diam}{diam}

\title[Genuinely sharp heat kernel estimates]
	{Genuinely sharp heat kernel estimates \\
	on compact rank-one symmetric spaces, \\
	for Jacobi expansions, \\ 
	on a ball and on a simplex}

\author[A. Nowak]{Adam Nowak}
\address{Adam Nowak, \newline
			Institute of Mathematics,
		Polish Academy of Sciences, \newline
      \'Sniadeckich 8,
      00--656 Warszawa, Poland    
      }
\email{anowak@impan.pl}

\author[P. Sj\"ogren]{Peter Sj\"ogren}
\address{Peter Sj\"ogren, \newline
			Mathematical Sciences, University of Gothenburg \newline
Mathematical Sciences, Chalmers University of Technology \newline
SE-412 96 G\"oteborg, Sweden 
      }
\email{peters@chalmers.se}

\author[T.Z. Szarek]{Tomasz Z. Szarek}
\address{Tomasz Z. Szarek,     \newline
Rutgers University,
Department of Mathematics, \newline
110  Frelinghuysen  Road, Piscataway, NJ 08854, USA \newline
\indent and \newline
University of Wroc\l aw, 
Mathematical Institute, \newline
Plac Grunwaldzki 2/4,
 50--384 Wroc\l aw, Poland
}
\email{tzs7@math.rutgers.edu}

\begin{document}

\begin{abstract}
We prove genuinely sharp two-sided global estimates for heat kernels on all compact rank-one symmetric spaces.
This generalizes the authors' recent result obtained for a Euclidean sphere of arbitrary dimension.
Furthermore, similar heat kernel bounds are shown in the context of classical Jacobi expansions,
on a ball and on a simplex. These results are more precise than the qualitatively sharp Gaussian estimates proved recently
by several authors.
\end{abstract}

\maketitle
\thispagestyle{empty}

\footnotetext{

Research supported by the National Science Centre of Poland within the project OPUS 2017/27/B/ST1/01623.
The third author was supported also by the Foundation for Polish Science via the START Scholarship.
}

\section{Description of results} \label{sec:desc}

Let $\mathbb{M}$ be a compact Riemannian symmetric space of rank one
or, which is the same (cf.\ \cite{He0}), a compact two-point homogeneous space.
Such spaces were completely classified by Wang \cite{Wa},
who showed that $\mathbb{M}$ is a Euclidean sphere or one of the projective spaces:
over real numbers or complex numbers or quaternions;
exceptionally $\mathbb{M}$ can be the Cayley projective plane over octonions.
The projective spaces are, in a sense, quite similar to the spheres. They have the property that
there exists a group of isometries acting transitively on pairs of equidistant points.
The theory and the geometry of these spaces are well known and understood, see e.g.\ Helgason's books \cite{He0,He1},
Gangolli \cite{Ga} and Sherman \cite{Sch}.
In fact, they can be regarded as model compact spaces of Riemannian geometry \cite{Ch}, which makes them significant
objects not only in mathematics but also in mathematical physics \cite{Sh}.
The associated analysis has been widely developed over many decades, and this is still an active area.
See for instance \cite{AzBa,BGM,CRS1,CRS2} and references therein for numerous examples of recent developments.

Let $K_t^{\mathbb{M}}(x,y)$ be the heat kernel associated to $\mathbb{M}$.
This kernel depends on $x$ and $y$ only through the geodesic distance $\dist(x,y)$.
Our main result Theorem \ref{thm:main} is the genuinely sharp bound
\begin{equation} \label{mbnd}
K_t^{\mathbb{M}}(x,y)\simeq \big[t+\diam\mathbb{M}-\dist(x,y)\big]^{-(d-\tilde{d}-1)/2}
	\frac{1}{t^{d/2}} \exp\bigg(-\frac{\dist^2(x,y)}{4t}\bigg),
\end{equation}
where $d$ is the dimension of $\mathbb{M}$ over the reals, and $\tilde{d}$ is the so-called antipodal dimension of $\mathbb{M}$
(see Section \ref{sec:sym}).
This bound holds uniformly in $x,y \in \mathbb{M}$ and $0 < t \le T$, with an arbitrary fixed $T<\infty$.
On the other hand, for $t \ge T$ one has $K_t^{\mathbb{M}}(x,y)\simeq 1$, which is a well-known fact.

In the simplest case, when $\mathbb{M}$ is a Euclidean sphere of arbitrary dimension, \eqref{mbnd}
has been proved only recently, by the authors \cite{NoSjSz}. For the remaining $\mathbb{M}$, 
our result improves significantly the best previously known global Gaussian bounds
$$
\frac{1}{t^{d/2}} \exp\bigg(-\frac{\dist^2(x,y)}{(4-\varepsilon)t}\bigg) 
\lesssim K_t^{\mathbb{M}}(x,y) \lesssim
\frac{1}{t^{d/2}} \exp\bigg(-\frac{\dist^2(x,y)}{(4+\varepsilon)t}\bigg),
\qquad 0 < t \le T,
$$
(here $\varepsilon > 0$ is arbitrary and fixed), which are only qualitatively sharp, in the sense that the constants appearing in the
exponential factors are different from each other and from $4$. These bounds are due to Li and Yau \cite{LY},
see \cite[Theorems 5.5.6 and 5.6.3, Note 5.6.4]{Da} or \cite[Chapter XII, Section 12]{Ch1}.
This is why we call our estimates genuinely sharp.

The heart of the problem here in passing from qualitatively to genuinely sharp estimates is to determine the relevant
polynomial factor, since the common exponential constant in \eqref{mbnd} is clear due to an asymptotic result of
Varadhan \cite{Va}, see e.g.\ \cite[Corollary 5.6.5]{Da}. This task is by no means trivial, as can be seen already from
the spherical heat kernel case. In fact, it is the geometry of $\mathbb{M}$ that has a decisive significance, and to handle
it one needs some advanced tools.

It is interesting to note that the proof of \eqref{mbnd} also gives genuinely sharp bounds for the derivative of $K_t^{\mathbb{M}}$
as a function of the geodesic distance, see Corollary \ref{cor:main}. In particular, this verifies the natural
conjecture that $K_t^{\mathbb{M}}(x,y)$ is a strictly decreasing function of $\dist(x,y)$. 
This monotonicity of the heat kernel was proved for several spaces invariant under rotation in \cite{ACMM}, and our result
enriches
the list of spaces treated there with the projective spaces mentioned.

Our proof of \eqref{mbnd} relies on reducing the problem, via a suitable addition formula for the eigenfunctions
of the Laplace-Beltrami operator on $\mathbb{M}$, to showing genuinely sharp bounds for the heat kernel related to
orthogonal expansions in Jacobi polynomials $P_n^{\a,\b}$, for certain discrete values of the parameters $\a,\b$.
This led us to proving genuinely sharp estimates for the Jacobi heat kernel 
for all $\a,\b \ge -1/2$; see Theorem \ref{thm:jac3}
which is no doubt of independent interest. For these $\a,\b$, it improves the qualitatively sharp
estimates obtained independently in \cite{CKP} and \cite{NoSj}, and for general $\a,\b > -1$ it obviously leads to
a conjecture about the optimal bounds.

Other classical frameworks intimately connected with Jacobi expansions are those of classical orthogonal
expansions in a Euclidean ball and on a simplex, cf.\ \cite{DaiXu,DuXu}. We take this opportunity to show
genuinely sharp heat kernel bounds in these settings for many values of the parameters involved;
see Theorems \ref{thm:heatball} and \ref{thm:heatsim}. All this
refines the very recent qualitatively sharp estimates by Kerkyacharian, Petrushev and Xu \cite{KPX1,KPX2}, and by two of the
authors \cite{SjSz}. Moreover, this hints how the optimal bounds for general parameters should look like.

Compared with qualitatively sharp estimates, genuinely sharp heat kernel bounds are in general harder to prove and
appear rarely in the literature; the case of the hyperbolic space \cite{DaMa} is one of these sparse instances.
The example of the Euclidean sphere \cite{NoSjSz}, as well as the present paper, show that this is a difficult
problem even for very simple Riemannian manifolds. In this connection, it is perhaps worth mentioning the recent papers
\cite{BM1,BM2,MS,MSZ} where such results were obtained for Dirichlet heat kernels related to Bessel operators in half-lines,
the Dirichlet heat kernel in Euclidean balls of arbitrary dimension, and the Fourier-Bessel heat kernel on the interval $(0,1)$.
This was achieved by a clever combination of probabilistic and analytic methods.
We note that the ball setting considered in \cite{MS} is not the same as the one in this paper.

The organization of the paper is as follows. In Section \ref{sec:prep} we introduce basic notation and prove the key
technical Lemma \ref{lem:FVII}. The next sections are devoted to estimates of the Jacobi heat kernel
(Section \ref{sec:jac}), the heat kernels on compact rank-one symmetric spaces (Section \ref{sec:sym}), the heat kernel
on the ball (Section \ref{sec:ball}) and the heat kernel on the simplex (Section \ref{sec:simplex}).

\section{Technical preparation} \label{sec:prep}

Throughout the paper we use a standard notation. The minimum and the maximum of two quantities will be indicated by
$\wedge$ and $\vee$, respectively.
Further, we will frequently use the notation $X \lesssim Y$ to indicate that
$X \le C Y$ with a positive constant $C$ independent of significant quantities. We shall write
$X \simeq Y$ when simultaneously $X \lesssim Y$ and $Y \lesssim X$. 

For a parameter $\nu \ge -1/2$, let $\Pi_{\nu}$ be the probability measure in $[-1,1]$ given by
\begin{align} \label{Pimeas}
d\Pi_{\nu}(w) = \frac{\Gamma(\nu+1)}{\sqrt{\pi} \, \Gamma(\nu+1/2)} \big(1-w^2\big)^{\nu-1/2}\, dw, \qquad \nu > -1/2,
\end{align}
and $\Pi_{-1/2} = (\delta_{-1}+\delta_1)/2$, the mean of Dirac deltas.

The main aim of this section is to prove the following technical result.
\begin{lem} \label{lem:FVII}
Let $\nu \ge -1/2$ and $\xi \in \R$ be fixed. Denote $\Phi_{A,B}(w) = \arccos(A+Bw)$. Then
\begin{align*}
& \int_{[0,1]} \big( \pi - \Phi_{A,B}(w) + D \big)^{-\xi} \exp\big(-\Phi^2_{A,B}(w)/D\big) \, d\Pi_{\nu}(w) \\
& \quad \simeq
D^{\nu+1/2} \big( \pi - \Phi_{A,B}(1) + D \big)^{-\xi}
\Big( B \big( \pi - \Phi_{A,B}(1) \big)^{-1} + D \Big)^{-\nu-1/2}
\exp \big(-\Phi^2_{A,B}(1)/D \big),
\end{align*}
uniformly in $0 \le B \le 1$, $-1 \le A \le 1 - B$ and $D > 0$;
here $B \big( \pi - \Phi_{A,B}(1) \big)^{-1}$ is understood as $0$ if $B=0$.
\end{lem}

To prove Lemma \ref{lem:FVII}, we need to estimate first a more elementary integral.
\begin{lem} \label{lem:F7}
Let $\g > -1$ be fixed. Then
\begin{align*}
\int_a^b e^{-x^2} (x-a)^{\g} x^{\g + 1} \, dx
\simeq
\bigg( \frac{(b-a)b}{(b-a)b + 1} \bigg)^{\g + 1} e^{-a^2}, 
\qquad 0 \le a \le b.
\end{align*}
\end{lem}

\begin{proof}
Denote the left-hand side in question by $I$. Changing the variable by $y = (x-a)(x + a)$, we get
\begin{align*}
I \simeq 
e^{-a^2} \int_0^{(b-a) (b+a)} e^{-y} y^{\g} \, dy
\simeq
e^{-a^2} \big[ (b-a) (b+a) \wedge 1 \big]^{\g + 1}
\simeq
e^{-a^2} \big[ (b-a) b \wedge 1 \big]^{\g + 1},
\end{align*}
since $x \simeq x+a$ uniformly in $x\in [a,b]$.
The conclusion follows.
\end{proof}

\begin{proof}[{Proof of Lemma \ref{lem:FVII}}]
The cases when $\nu = -1/2$ or $B=0$ are trivial, so we assume that $\nu > -1/2$ and $B > 0$.
Denote the integral in the statement by $I$ and set
\begin{align*}
\vp_0 = \Phi_{A,B}(0) \qquad \textrm{and} \qquad \vp_1 = \Phi_{A,B}(1),
\end{align*}
so that $0\le \vp_1 < \vp_0 \le \pi$. Then $A=\cos\vp_0$ and $B=\cos\vp_1-\cos\vp_0$.
Changing the variable by $A + Bw = \cos \psi$, we get
\begin{align*}
I \simeq 
B^{-\nu-1/2}
\int_{\vp_1}^{\vp_0} ( \pi - \psi + D )^{-\xi} \exp\big(-\psi^2/D\big)  
\sin \psi \, \big(\cos \vp_1 - \cos \psi \big)^{\nu - 1/2} \, d\psi.
\end{align*}
Using basic trigonometric identities, we see that 
$\cos \vp_1 - \cos \psi \simeq (\psi - \vp_1) \psi (\pi - \vp_1 )$ and
$\sin \psi \simeq \psi (\pi - \psi)$ for $\psi \in (\vp_1, \vp_0 )$.
This together with the change of variable $\psi \mapsto \psi \sqrt{D}$ gives us
$$
I \simeq B^{-\nu-1/2} (\pi - \vp_1)^{\nu - 1/2} D^{\nu+1/2 + (1-\xi)/2} J,
$$
where
$$
J=
\int_{\vp_1/\sqrt{D}}^{\vp_0/\sqrt{D}} 
\big( \pi/\sqrt{D} - \psi + \sqrt{D} \big)^{-\xi}
\big( \psi - \vp_1/\sqrt{D} \big)^{\nu - 1/2} \psi^{\nu+1/2}
\big( \pi/\sqrt{D} - \psi \big) \exp\big(-\psi^2\big) \, d\psi.
$$
Let
$$
Q =
D^{ (\xi - 1)/2} (\pi - \vp_1 + D)^{- \xi} (\pi - \vp_1)
\bigg( \frac{ (\vp_0 - \vp_1 ) \vp_0 }{ (\vp_0 - \vp_1 ) \vp_0 + D} \bigg)^{\nu+1/2}
\exp\big(-\vp_1^2/D\big).
$$
Using the fact that $ B = \cos \vp_1 - \cos \vp_0 \simeq (\vp_0 - \vp_1 ) \vp_0 (\pi - \vp_1 )$,
we see that in order to finish the proof of Lemma~\ref{lem:FVII} it suffices to show that $J \simeq Q$.

To proceed, we define $\widetilde{\vp} = (\pi + \vp_1)/2 \in [\pi/2,\pi)$ and we split the region of integration in $J$ into 
$(\vp_1/\sqrt{D}, (\widetilde{\vp} \wedge \vp_0)/\sqrt{D})$ and 
$( (\widetilde{\vp} \wedge \vp_0)/\sqrt{D}, \vp_0/\sqrt{D} )$, denoting the corresponding integrals by $J_1$ and $J_2$, respectively. 
Note that $J_2$ vanishes in some cases.

We first treat $J_1$. Observe that for $\psi \in (\vp_1/\sqrt{D} , \widetilde{\vp}/\sqrt{D})$ we have 
\begin{align*}
(\pi - \vp_1)/(2\sqrt{D})
< 
\pi/\sqrt{D} - \psi 
< 
(\pi - \vp_1)/\sqrt{D}.
\end{align*}
Using this and then Lemma~\ref{lem:F7} specified to $\g = \nu - 1/2$, 
$a = \vp_1/\sqrt{D}$ and $b = (\widetilde{\vp} \wedge \vp_0)/\sqrt{D}$, together with the relations
\begin{align*}
\widetilde{\vp} \wedge \vp_0 - \vp_1
=
(\widetilde{\vp} - \vp_1) \wedge (\vp_0 - \vp_1)
=
(\pi-\vp_1)/2 \wedge (\vp_0 - \vp_1)
\simeq
\vp_0 - \vp_1, \qquad
\widetilde{\vp} \wedge \vp_0
\simeq
\vp_0,
\end{align*}
we get
\begin{align*}
J_1 
& \simeq
D^{ (\xi - 1)/2} (\pi - \vp_1 + D)^{- \xi} (\pi - \vp_1)
\int_{\vp_1/\sqrt{D}}^{(\widetilde{\vp} \wedge \vp_0)/\sqrt{D}} 
\big( \psi - \vp_1/\sqrt{D} \big)^{\nu - 1/2} \psi^{\nu+1/2}
\exp\big(-\psi^2\big) \, d\psi \\
& \simeq
D^{ (\xi - 1)/2} (\pi - \vp_1 + D)^{- \xi} (\pi - \vp_1)
\bigg( \frac{ (\vp_0 - \vp_1 ) \vp_0 }{ (\vp_0 - \vp_1 ) \vp_0 + D} \bigg)^{\nu+1/2}
\exp\big(-\vp_1^2/D\big) = Q.
\end{align*}

Therefore we have shown that $J_1 \simeq Q$, and in order to finish the proof of Lemma~\ref{lem:FVII}
it is enough to verify that $J_2 \lesssim Q$. To do so, it is convenient to distinguish two cases.

\noindent \textbf{Case 1:} $\pi - \vp_1 \le D$. Here we have
\begin{align*}
\pi/\sqrt{D} - \psi < (\pi - \vp_1)/\sqrt{D} \le \sqrt{D}, \qquad
\psi \in \big( \vp_1/\sqrt{D}, \vp_0/\sqrt{D} \big),
\end{align*}
and consequently 
\begin{align*}
J_2 \le J \lesssim
D^{ -(\xi + 1)/2} (\pi - \vp_1) 
\int_{\vp_1/\sqrt{D}}^{ \vp_0/\sqrt{D}} 
\big( \psi - \vp_1/\sqrt{D} \big)^{\nu - 1/2} \psi^{\nu+1/2}
\exp\big(-\psi^2\big) \, d\psi. 
\end{align*}
Applying now Lemma~\ref{lem:F7} with $\g = \nu - 1/2$, 
$a = \vp_1/\sqrt{D}$ and $b = \vp_0/\sqrt{D}$, we obtain 
\begin{align*}
J_2 \lesssim
D^{ -(\xi + 1)/2} (\pi - \vp_1) 
\bigg( \frac{ (\vp_0 - \vp_1 ) \vp_0 }{ (\vp_0 - \vp_1 ) \vp_0 + D} \bigg)^{\nu+1/2}
\exp\big(-\vp_1^2/D\big) \simeq Q.
\end{align*}
This finishes Case 1.

\noindent \textbf{Case 2:} $\pi - \vp_1 > D$. Here we may assume that 
$\widetilde{\vp} < \vp_0$ since otherwise $J_2=0$. This means that 
$0 \le \vp_1 < \widetilde{\vp} < \vp_0 \le \pi$, and we also have $\widetilde{\vp} \ge \pi/2$.
Further, observe that the condition $\pi - \vp_1 > D$ forces
\begin{align*}
\vp_0 - \vp_1 > \widetilde{\vp} - \vp_1 = (\pi - \vp_1)/2 > D/2.
\end{align*}
This together with the fact that $\pi/2 < \vp_0 \le \pi$ implies that
\begin{align} \label{estQ1}
Q \simeq D^{ (\xi - 1)/2} (\pi - \vp_1)^{1 - \xi}
\exp\big(-\vp_1^2/D\big).
\end{align}
Further, observe that for 
$\psi \in \big( \widetilde{\vp}/\sqrt{D}, \vp_0/\sqrt{D} \big)$ we have
\begin{align*}
& (\pi - \vp_1)/(2\sqrt{D})
 < 
\psi - \vp_1/\sqrt{D}
< 
(\vp_0 - \vp_1)/\sqrt{D}
< 
(\pi - \vp_1)/\sqrt{D}, \\ 
& \pi/\sqrt{D} - \psi 
< 
(\pi - \vp_1)/(2\sqrt{D}), \qquad
\psi \simeq 1/\sqrt{D}.
\end{align*}
Using these relations and then changing the variable 
$\pi/\sqrt{D} - \psi + \sqrt{D} \mapsto s$, we arrive at
\begin{align*}
J_2 &\lesssim
\Big( \frac{ \pi - \vp_1  }{ D} \Big)^{\nu + 1/2}
\exp\bigg(-\frac{\widetilde{\vp}^{\,2}}{D}\bigg)
\int_{\widetilde{\vp}/\sqrt{D}}^{ \vp_0/\sqrt{D}} 
\big( \pi/\sqrt{D} - \psi + \sqrt{D} \big)^{-\xi} \, d\psi \\
& \le 
\Big( \frac{ \pi - \vp_1  }{ D} \Big)^{\nu + 1/2}
\exp\bigg(-\frac{\widetilde{\vp}^{\,2}}{D}\bigg)
\int_{\sqrt{D}}^{ (\pi - \vp_1)/(2\sqrt{D}) + \sqrt{D} } 
s^{-\xi} \, ds.
\end{align*}
Combining this estimate with \eqref{estQ1} and the relations
\begin{align*}
\widetilde{\vp}^{\,2} - \vp_1^2 
= (\widetilde{\vp} - \vp_1)(\widetilde{\vp} + \vp_1)
\ge 
\pi (\pi - \vp_1)/4
\ge 
(\pi - \vp_1)/2,
\end{align*}
we see that in order to prove that $J_2 \lesssim Q$ it is enough to show that 
\begin{align*}
\Big( \frac{ \pi - \vp_1  }{ D} \Big)^{\nu + 1/2}
\Big( \frac{ \pi - \vp_1  }{ \sqrt{D} } \Big)^{\xi - 1}
\exp\Big(-\frac{\pi - \vp_1}{2D} \Big)
\int_{\sqrt{D}}^{ (\pi - \vp_1)/(2\sqrt{D}) + \sqrt{D} } 
s^{-\xi} \, ds
\lesssim 1.
\end{align*}
Evaluating the last integral and using the fact that 
$\big( \frac{ \pi - \vp_1  }{ D} \big)^{\nu + 1/2} \exp\big(-\frac{\pi - \vp_1}{4D} \big) \lesssim 1$,
we can further reduce our task to showing that 
\begin{align*}
\Big( \frac{ \pi - \vp_1  }{ \sqrt{D} } \Big)^{\xi - 1}
\exp\Big(-\frac{\pi - \vp_1}{4D} \Big) 
\left.
\begin{cases}
D^{(1-\xi)/2}, & \xi > 1  \\
\log\left(1 + \frac{\pi-\vp_1}{2D} \right), &  \xi = 1 \\
\big( \frac{ \pi - \vp_1  }{ \sqrt{D} } \big)^{-\xi + 1}, & \xi < 1
\end{cases} \right\}
\lesssim 1.
\end{align*}
This, however, can easily be justified.

Case 2 is finished, and the proof of Lemma~\ref{lem:FVII} is complete.
\end{proof}

\section{Estimates of the Jacobi heat kernel} \label{sec:jac}

Let $\a,\b > -1$.
Recall that the heat kernel related to classical Jacobi expansions is given by (see e.g.\ \cite{NoSj})
\begin{align} \label{iden3}
G_t^{\a,\b}(x,y) = \sum_{n=0}^{\infty} e^{-t\lambda^{\a,\b}_n} 
	\frac{P_n^{\a,\b}(x) P_n^{\a,\b}(y)}{h_n^{\a,\b}},
	\qquad x,y \in [-1,1], \quad t > 0,
\end{align}
where $\lambda^{\a,\b}_n = n(n + \a+\b+1)$ and
$h_n^{\a,\b} = \|P_n^{\a,\b}\|_{2,(\a,\b)}^2$, the norm taken in $L^2$ of the interval $[-1,1]$ with
the measure $d\rho_{\a,\b}(x) = (1-x)^{\a}(1+x)^{\b} dx$.
Here $\{P_n^{\a,\b}:n=0,1,2,\ldots\}$ is the system of Jacobi polynomials (cf.\ \cite{Sz}), which constitutes
an orthogonal basis in $L^2([-1,1],d\rho_{\a,\b})$.

For large $t$, say $t \ge 1$, the Jacobi heat kernel is essentially constant, see \cite{CKP} or \cite{NoSj},
\begin{equation} \label{jac:long}
G_t^{\a,\b}(x,y) \simeq 1, \qquad x,y \in [-1,1], \quad t \ge 1.
\end{equation}
Qualitatively sharp bounds of $G_t^{\a,\b}(x,y)$ for small $t$ were obtained in \cite{CKP},
and by completely different methods in \cite{NoSj} with the restriction $\a,\b \ge -1/2$.
The following genuinely sharp estimate was conjectured in \cite{NoSjSz} for all $\a,\b > -1$:
\begin{align} \nonumber
& G_t^{\a,\b}(\cos\varphi,\cos\psi)  \\ 
& \quad \simeq \big[ t+ \varphi\psi\big]^{-\a-1/2}
	\big[ t + (\pi-\varphi)(\pi-\psi)\big]^{-\b-1/2} \frac{1}{\sqrt{t}} \exp\bigg( -\frac{(\varphi-\psi)^2}{4t}\bigg),
	\qquad 0 < t \le 1,
	\label{jac:gen}
\end{align}
uniformly in $\varphi,\psi \in [0,\pi]$. An important special case of \eqref{jac:gen} is
\begin{equation} \label{jac:spec}
G_t^{\lambda,\lambda}(\cos\varphi,1) \simeq t^{-\lambda-1/2} \big[ t+\pi-\varphi \big]^{-\lambda-1/2}
	\frac{1}{\sqrt{t}} \exp\bigg( -\frac{\varphi^2}{4t}\bigg), \qquad 0 < t \le 1,
\end{equation}
where $\lambda > -1$ and $\varphi$ is as before.
The bound \eqref{jac:spec} was proved in \cite{NoSjSz} for $\lambda \ge -1/2$ half-integer.
Here we will generalize that result, as in \cite{NoSjSz} by means of an analytic approach instead of the abstract
technology employed in \cite{CKP}. Our main tool will be the following reduction formula from \cite[Theorem 3.1]{NoSj}:
For $\a,\b \ge -1/2$,
\begin{align} \nonumber
& G_t^{\a,\b}(\cos\varphi,\cos\psi) \\ & = \mathcal{C}_{\a,\b} \iint G_{t/4}^{\a+\b+1/2,\a+\b+1/2}
	\Big( u \sin\frac{\varphi}2\sin\frac{\psi}2 + v \cos\frac{\varphi}2 \cos\frac{\psi}2,1\Big)\, d\Pi_{\a}(u)d\Pi_{\b}(v),
	\label{red}
\end{align}
with $\mathcal{C}_{\a,\b} = \sqrt{\pi} \, \Gamma(\a+\b+3/2)/(2^{\a+\b+1}\Gamma(\a+1)\Gamma(\b+1))$.
This is a consequence of a product formula due to Dijksma and Koornwinder \cite{DK}.

We now formulate our main result in the Jacobi setting.
\begin{thm} \label{thm:jac3}
The bound \eqref{jac:gen} holds for any pair $\a,\b \ge -1/2$. 
\end{thm}

\begin{rem} \emph{
The estimate \eqref{jac:gen} holds for $\b=1/2$ and all $\a > -1$, as well as for $\a=1/2$ and all $\b > -1$.
This can be deduced from genuinely sharp Fourier-Bessel heat kernel estimates obtained in \cite{MSZ} and
a connection between Jacobi and Fourier-Bessel settings established in \cite{NR}.
The same connection and Theorem \ref{thm:jac3} provide an alternative and fully analytic proof of the genuinely
sharp Fourier-Bessel heat kernel bounds, when the Fourier-Bessel parameter $\nu$ satisfies $\nu \ge -1/2$.}
\end{rem}

To prove Theorem \ref{thm:jac3}, we need several auxiliary results. The first one is a special case of the comparison
principle \cite[Theorem 3.5]{NoSj}, where we take $\epsilon=0$. With this
choice, one can delete the assumption $\a \ge -\epsilon/2$, as can
be seen from the proof of \cite[Theorem 3.5]{NoSj} (this slight enhancement was already observed in \cite[p.\,346]{NSS}).
\begin{lem} \label{lem:comp}
Let $\a,\b > -1$ and assume that $\b \ge -\delta/2$ for some $\delta \ge 0$. Then
$$
\big[ (1+x)(1+y)\big]^{\delta/2} G_t^{\a,\b+\delta}(x,y) \le e^{\delta (\a+\b+1+\delta/2)t/2} G_t^{\a,\b}(x,y),
	\qquad x,y \in [-1,1], \quad t > 0.
$$
\end{lem}

In the next lemma, we restrict  $\vp$  to  $[0, \pi/2]$.

\begin{lem} \label{lem:us}
If \eqref{jac:gen} with $\psi=0$ holds uniformly in $\varphi \in [0,\pi/2]$ whenever $\a,\b \ge -1/2$ and
$\a+\b$ is half-integer, then \eqref{jac:spec} holds uniformly in $\varphi \in [0,\pi/2]$ whenever $\lambda > 0$.
\end{lem}

\begin{proof}
By Lemma \ref{lem:comp} we have
$$
G_t^{\a,\b+\delta}(\cos\varphi,1) \lesssim
G_t^{\a,\b}(\cos\varphi,1), \qquad \varphi \in [0,\pi/2], \quad 0 < t \le 1,
$$
provided that $\b \ge -\delta/2$ and $\delta \ge 0$;
the implicit multiplicative constant here depends only on $\a,\b$ and $\delta$.

Let $\lambda > 0$ and $N \in \{0,1,2,\ldots \}$ be such that $N/2 < 2\lambda \le N/2+1/2$.
Taking above $\a=\lambda$ and either $\b=\lambda$, $\delta=N/2+1/2-2\lambda$ or $\b=\lambda-\delta$
with $\delta = 2\lambda-N/2$ (notice that in both cases $\b > -1$, $\b \ge -\delta/2$ and $\delta \ge 0$), we obtain
$$
G_{t}^{\lambda,N/2+1/2-\lambda}(\cos\varphi,1) \lesssim G_{t}^{\lambda,\lambda}(\cos\varphi,1) \lesssim
	G_{t}^{\lambda,N/2-\lambda}(\cos\varphi,1), \qquad \varphi \in [0,\pi/2], \quad 0 < t \le 1,
$$
where the implicit multiplicative constants depend only on $\lambda$.
From here the conclusion easily follows.
\end{proof}

\begin{lem} \label{lem:spec}
Let $\lambda > -1/2$ be fixed. If \eqref{jac:spec} holds uniformly in $\varphi \in [0,\pi/2]$, then
it also holds uniformly in $\varphi \in [0,\pi]$ with $\lambda$ replaced by $\lambda'=\lambda/2-1/4$.
\end{lem}

\begin{rem} \emph{
Post factum, Lemma \ref{lem:spec} makes it possible to simplify the reasonings in \cite{NoSj} and \cite{NoSjSz}. The crucial
point is that one can start with the restricted range of $\varphi$ and thus avoid the most delicate situation
when the arguments of the kernel are near opposite endpoints of $[-1,1]$.}
\end{rem}

\begin{proof}[{Proof of Lemma \ref{lem:spec}}]
We consider $0 < t \le 1$ and $\varphi \in [0,\pi]$. From the reduction formula \eqref{red} we get
$$
G_t^{\lambda',\lambda'}(\cos\varphi,1) \simeq
	\int G_{t/4}^{\lambda,\lambda}\Big(v \cos\frac{\varphi}2,1\Big) \,d\Pi_{\lambda'}(v).
$$
The main contribution to the integral here comes from the integration
over $[0,1]$. This follows from the symmetry of $d\Pi_{\lambda'}$ and the inequality
$$
G_{t}^{\lambda,\lambda}(-x,1) \le G_t^{\lambda,\lambda}(x,1), \qquad x \in [0,1], \quad t > 0,
$$
which is a direct consequence of the second identity in Lemma \ref{lem:qiden} below
and of the positivity of the Jacobi heat kernel.
Thus
$$
G_t^{\lambda',\lambda'}(\cos\varphi,1) \simeq
	\int_{[0,1]} G_{t/4}^{\lambda,\lambda}\Big(v \cos\frac{\varphi}2,1\Big) \,d\Pi_{\lambda'}(v).
$$

Let $f(v,\varphi) = \arccos(v\cos\frac{\varphi}2)$.
Plugging in the assumed bound \eqref{jac:spec} for $G_{t/4}^{\lambda,\lambda}$ in the integral above
(notice that for $v \in [0,1]$ and $\varphi \in [0,\pi]$ one has $f(v,\varphi) \in [0,\pi/2]$),
we see that we need only show that
\begin{align*}
 \int_{[0,1]} 
	\exp\bigg({-\frac{f^2(v,\varphi)}{t}}\bigg) \, d\Pi_{\lambda'}(v)
 \simeq t^{\lambda'+1/2} (t+\pi-\varphi)^{-\lambda'-1/2} \exp\bigg(-\frac{\varphi^2}{4t}\bigg),
\end{align*}
uniformly in $\varphi \in [0,\pi]$ and $0 < t \le 1$.
This, however, is a direct consequence of Lemma \ref{lem:FVII} applied with $\xi=0$, $\nu=\lambda'$,
$A=0$, $B=\cos\frac{\varphi}2$ and $D=t$.
\end{proof}

\begin{lem} \label{lem:gen}
Let $\a,\b \ge -1/2$ be fixed. If \eqref{jac:spec} holds uniformly in $\varphi \in [0,\pi]$ with
$\lambda = \a+\b+1/2$, then \eqref{jac:gen} holds uniformly in $\varphi,\psi \in [0,\pi]$.
\end{lem}

\begin{proof}
Let 
$f(u,v,\vp,\psi) = \arccos( u \sin\frac{\varphi}2\sin\frac{\psi}2 + v \cos\frac{\varphi}2 \cos\frac{\psi}2)$.
Combining \eqref{red} with the assumed relation \eqref{jac:spec}, we see that our task is reduced to showing that
\begin{align*}
& \iint 
\big[ t+ \pi - f(u,v,\vp,\psi) \big]^{-\a - \b -1} 
\exp\bigg( -\frac{f^2(u,v,\vp,\psi)}{t}\bigg) \, d\Pi_{\a}(u)d\Pi_{\b}(v) \\
& \quad \simeq
t^{\a + \b + 1} \big( t+ \varphi\psi\big)^{-\a-1/2}
	\big[ t + (\pi-\varphi)(\pi-\psi)\big]^{-\b-1/2} 
\exp\bigg( -\frac{(\varphi-\psi)^2}{4t}\bigg),
\end{align*}
uniformly in $\varphi,\psi \in [0,\pi]$ and $0 < t \le 1$.

We first observe that the main contribution in the integral above comes from the integration over $[0,1]^2$.
Indeed, this can be  verified by reflecting $u \mapsto - u$ and/or 
$v \mapsto - v$ and using the symmetry of the measures $d\Pi_{\a}$ and $d\Pi_{\b}$ and then the estimate
\begin{align} \label{estZ}
(t + \pi - \t )^{\gamma} \exp\bigg( -\frac{\t^2}{4t} \bigg)
\lesssim 
(t + \pi - \tau )^{\gamma} \exp\bigg( -\frac{\tau^2}{4t} \bigg),
\qquad 0 \le \tau \le \t \le \pi, \quad t > 0,
\end{align}
where $\gamma \in \R$ is fixed. Here we let $\gamma = - \a -\b - 1$, $\t = f(\pm u,\pm v, \vp, \psi)$,
$\tau = f( u, v, \vp, \psi)$ with $u,v \in [0,1]$ and $\vp,\psi \in [0,\pi]$.
To justify this estimate, it is convenient to distinguish the two cases $\t \le 3\pi/4$ and $\t > 3\pi/4$. 
We omit elementary details.

Further, using the fact that $f(u,v,\vp,\psi) \in [0,\pi/2]$ for $u,v \in [0,1]$ and $\vp,\psi \in [0,\pi]$,
we see that to finish the proof of Lemma \ref{lem:gen} it is enough to prove that
\begin{align} \nonumber
& \iint_{[0,1]^2} 
\exp\bigg( -\frac{f^2(u,v,\vp,\psi)}{t}\bigg) \, d\Pi_{\a}(u)d\Pi_{\b}(v) \\ \label{red2}
& \qquad \qquad \simeq
t^{\a + \b + 1} \big( t+ \varphi\psi\big)^{-\a-1/2}
	\big[ t + (\pi-\varphi)(\pi-\psi)\big]^{-\b-1/2} 
\exp\bigg( -\frac{(\varphi-\psi)^2}{4t}\bigg),
\end{align}
uniformly in $\varphi,\psi \in [0,\pi]$ and $0 < t \le 1$.

Denote the double integral in \eqref{red2} by $I$.
Applying Lemma~\ref{lem:FVII} to the integral with respect to $v$
(with $\xi=0$, $\nu = \b$, $A = u \sin\frac{\varphi}2\sin\frac{\psi}2$, 
$B = \cos\frac{\varphi}2 \cos\frac{\psi}2$ and $D=t$), we infer that
\begin{align*}
I \simeq
t^{ \b + 1/2 } \bigg( t + \cos\frac{\varphi}2 \cos\frac{\psi}2 \bigg)^{-\b -1/2} \int_{[0,1]} 
\exp\bigg( -\frac{f^2(u,1,\vp,\psi)}{t}\bigg) \, d\Pi_{\a}(u).
\end{align*}
Then another application of Lemma~\ref{lem:FVII} (this time specified to 
$\xi=0$, $\nu = \a$, $A = \cos\frac{\varphi}2 \cos\frac{\psi}2$, 
$B = \sin\frac{\varphi}2\sin\frac{\psi}2$ and $D=t$) produces
\begin{align*}
I \simeq
t^{ \a + \b + 1 } 
\bigg( t + \cos\frac{\varphi}2 \cos\frac{\psi}2 \bigg)^{-\b -1/2}
\bigg( t + \sin\frac{\varphi}2\sin\frac{\psi}2 \bigg)^{-\a -1/2}
\exp\bigg( -\frac{(\varphi-\psi)^2}{4t}\bigg).
\end{align*}
Since $\sin \t \simeq \t$ and $\cos \t \simeq \pi/2 - \t$ for $\t \in [0,\pi/2]$, we get \eqref{red2}.
\end{proof}

We are now in a position to prove Theorem \ref{thm:jac3}.
\begin{proof}[{Proof of Theorem \ref{thm:jac3}}]
For the sake of clarity, we divide the proof into several steps.
Recall that in \cite{NoSjSz} we proved that \eqref{jac:spec} holds uniformly in $\varphi \in [0,\pi]$
whenever $\lambda \ge -1/2$ is a half-integer.
This is our starting point in the present proof, call it STEP 0.

STEP 1. Using STEP 0, we infer from Lemma \ref{lem:gen} that \eqref{jac:gen} holds uniformly in
$\varphi,\psi \in [0,\pi]$ whenever the sum $\a+\b$ is a half-integer. In particular, we can take $\psi=0$
and restrict $\varphi$ to $[0,\pi/2]$.

STEP 2. From STEP 1 and Lemma \ref{lem:us}, we conclude that \eqref{jac:spec} holds uniformly in $\varphi \in [0,\pi/2]$
whenever $\lambda > 0$.

STEP 3. Use the result of STEP 2 and Lemma \ref{lem:spec} to see that \eqref{jac:spec} holds uniformly in
$\varphi \in [0,\pi]$
whenever $\lambda \ge -1/2$ (to cover $\lambda$ close to $-1/2$ Lemma \ref{lem:spec} must be iterated sufficiently
many times; the case $\lambda=-1/2$ is covered by STEP 0).

STEP 4. Combine the result of STEP 3 with Lemma \ref{lem:gen} to obtain that \eqref{jac:gen} holds uniformly in
$\varphi,\psi \in [0,\pi]$ whenever $\a,\b \ge -1/2$. This completes the proof.
\end{proof}

We finish  this section with two simple technical results in the Jacobi setting. Both of
them are most probably known, at least as folklore.
\begin{lem} \label{lem:diff}
Let $\a,\b > -1$. Then
$$
\frac{d}{dx} G_t^{\a,\b}(x,1) = 2 (\a+1) e^{-t(\a+\b+2)} G_t^{\a+1,\b+1}(x,1), \qquad x \in [-1,1], \quad t > 0.
$$
\end{lem}

The ultraspherical case $\a=\b$ of this lemma can be found e.g.\ in \cite{BGL}.
A similar formula in the context of compact Riemannian manifolds is due to Millson,
cf.\ \cite[Chapter VI, Section 3]{Ch1}.

\begin{proof}[{Proof of Lemma \ref{lem:diff}}]
Combine \eqref{iden3} with the well-known formulas (see \cite[Chapter IV]{Sz})
\begin{align}
h_n^{\a,\b} & = \frac{2^{\a+\b+1} \Gamma(n+\a+1)\Gamma(n+\b+1)}{(2n+\a+\b+1)\Gamma(n+\a+\b+1)\Gamma(n+1)}, \label{h_for} \\
P_n^{\a,\b}(1) & = \frac{\Gamma(n+\a+1)}{\Gamma(n+1)\Gamma(\a+1)}, \label{P1_for} \\
\frac{d}{dx} P_n^{\a,\b}(x) & = \frac{1}2 (n+\a+\b+1) P_{n-1}^{\a+1,\b+1}(x), \nonumber
\end{align}
with suitable understanding of the first one when $n=0=\a+\b+1$
(in this case the product $(2n+\alpha+\beta+1) \Gamma(n+\alpha+\beta+1)$ must be replaced by $\Gamma(\alpha+\beta+2)$),
and with the convention that $P_{-1}^{\a,\b}\equiv 0$ in the third one.
Differentiating the series in \eqref{iden3} termwise is possible thanks to suitable control of the growth of Jacobi
polynomials for large $n$, see \cite[(7.32.2)]{Sz} or, for instance, \cite[(8)]{NoSj}.
\end{proof}

\begin{lem} \label{lem:qiden}
Let $\a > -1$. Then for $x,y \in [-1,1]$ and $t > 0$
\begin{align*}
G_t^{\a,-1/2} (2x^2 - 1, 2y^2 - 1) & = 
2^{-\a - 3/2} \big[ G_{t/4}^{\a,\a} (x,y) + G_{t/4}^{\a,\a} (-x,y) \big], \\
G_t^{\a,1/2} (2x^2 - 1, 2y^2 - 1) & = 2^{-\a - 5/2} e^{t(\a+1)/2} (xy)^{-1}
	\big[ G_{t/4}^{\a,\a} (x,y) - G_{t/4}^{\a,\a} (-x,y) \big].
\end{align*}
\end{lem}

It is worth noting that, in view of this lemma, the ultraspherical heat kernel
$G_t^{\a,\a}$ can be expressed in a relatively simple way via $G_{4t}^{\a,-1/2}$ and $G_{4t}^{\a,1/2}$.

\begin{proof}[{Proof of Lemma \ref{lem:qiden}}]
We invoke the quadratic transformations (cf.\ \cite[(4.1.5)]{Sz})
\begin{align*}
P_{2n}^{\alpha,\alpha}(x) & = \frac{\Gamma(2n+\alpha+1)\Gamma(n+1)}{\Gamma(n+\alpha+1)\Gamma(2n+1)}
	P_n^{\alpha,-1\slash 2}(2x^2-1), \\
P_{2n+1}^{\alpha,\alpha}(x) & = \frac{\Gamma(2n+\alpha+2)\Gamma(n+1)}{\Gamma(n+\alpha+1)\Gamma(2n+2)}\,
	xP_n^{\alpha,1\slash 2}(2x^2-1).
\end{align*}
Using this together with \eqref{h_for} and then performing some computations with the aid of the duplication
formula $\Gamma(z)\Gamma(z+1\slash 2) = \sqrt{\pi}\, 2^{-2z+1}\Gamma(2z)$, one finds that
\begin{align*}
\frac{P_{2n}^{\alpha,\alpha}(x)}{\|P_{2n}^{\alpha,\alpha}\|_{2,(\alpha,\alpha)}} & =
	\sqrt{2^{\alpha+1\slash 2}} \,
	\frac{P_n^{\alpha,-1\slash 2}(2x^2-1)}{\|P_n^{\alpha,-1\slash 2}\|_{2,(\alpha,-1\slash 2)}}, \\
\frac{P_{2n+1}^{\alpha,\alpha}(x)}{\|P_{2n+1}^{\alpha,\alpha}\|_{2,(\alpha,\alpha)}} & =
	\sqrt{2^{\alpha+3\slash 2}} \,
	\frac{xP_n^{\alpha,1\slash 2}(2x^2-1)}{\|P_n^{\alpha,1\slash 2}\|_{2,(\alpha,1\slash 2)}}.
\end{align*}
This together with \eqref{iden3} and the fact that the eigenvalues are related by
$$
\lambda_{2n}^{\alpha,\alpha} = 4\lambda_n^{\alpha,-1\slash 2}, \qquad
	\lambda_{2n+1}^{\alpha,\alpha} = 4\lambda_n^{\alpha,1\slash 2} + 2\alpha+2,
$$
gives the conclusion.
\end{proof}

\section{Heat kernel bounds on compact rank-one symmetric spaces} \label{sec:sym}

Let $\mathbb{M}$ be a compact symmetric space of rank one, viz., a compact two-point homogeneous space.
Thus $\mathbb{M}$ is a complete connected Riemannian manifold, with no boundary
and strictly positive sectional (and Ricci) curvature.
We denote by $d_{\mathbb{M}}$ the Riemannian geodesic distance on $\mathbb{M}$,
and by $\Delta_{\mathbb{M}}$ the Laplace-Beltrami operator. The Riemannian volume measure is denoted by $d\omega$.
The dimension of $\mathbb{M}$ will often be indicated by a superscript, i.e., $\mathbb{M}^d$ means
that $\mathbb{M}$ has dimension $d$ over the reals.

It is well known that for any $d \ge 1$, all possible $\mathbb{M}^d$ are given by the following list, see e.g. \cite{He1,Wa}:
\begin{itemize}
\item[(i)] Euclidean unit spheres $S^d$, $d=1,2,3,\ldots$
\item[(ii)] real projective spaces $\mathbb{P}^d(\mathbb{R})$, $d=2,3,4,\ldots$
\item[(iii)] complex projective spaces $\mathbb{P}^d(\mathbb{C})$, $d=4,6,8,\ldots$
\item[(iv)] quaternionic projective spaces $\mathbb{P}^d(\mathbb{H})$, $d=8,12,16,\ldots$
\item[(v)] the Cayley projective plane $\mathbb{P}^d(\mathbb{C}\textrm{ay})$, $d=16$.
\end{itemize}
Here $\mathbb{H}$ stands for the ring of quaternions, and $\mathbb{C}\textrm{ay}$ for the algebra of Cayley's octonions.
In (ii)--(iv) the lowest dimensions are omitted, since $\mathbb{P}^1(\mathbb{R}) = S^1$,
$\mathbb{P}^2(\mathbb{C}) = S^2$ and $\mathbb{P}^4(\mathbb{H})=S^4$.
Useful models of the spaces (i)--(v) can be found in \cite{Sh}.

Given a fixed point in $\mathbb{M}$, the set of points in $\mathbb{M}$ whose geodesic distance from it
is equal to the diameter of
$\mathbb{M}$ forms a submanifold of $\mathbb{M}$, which is called the antipodal manifold. In the case of $S^d$ it is trivial and
consists of only one point, while in the other cases (ii)--(v) the antipodal manifolds are identified, via isometric
isomorphisms, with
$\mathbb{P}^{d-1}(\mathbb{R})$, $\mathbb{P}^{d-2}(\mathbb{C})$, $\mathbb{P}^{d-4}(\mathbb{H})$ and $S^8$, respectively.
This was originally proved by Cartan \cite{Ca} and Nagano \cite{Na}, see also \cite{He0} and \cite{Ti} where the result is
re-obtained.
In our notation, the real dimensions of the corresponding antipodal manifolds will be indicated
by second superscripts, thus $\mathbb{M}^{d,\tilde{d}}$. 
We call $\tilde{d}$ the antipodal dimension.
Thus for $\mathbb{M}^d$ from the classes (i)--(v) we have, respectively,
$\mathbb{M}^{d,0}$, $\mathbb{M}^{d,d-1}$, $\mathbb{M}^{d,d-2}$, $\mathbb{M}^{d,d-4}$ and
$\mathbb{M}^{16,8}$.
Notice that the ``co-dimension'' $d-\tilde{d} \in \{d,1,2,4,8\}$
determines to which class (i)--(v) $\mathbb{M}^{d,\tilde{d}}$ belongs.

The operator $-\Delta_{\mathbb{M}}$, acting initially on $C^{\infty}(\mathbb{M})$, is essentially self-adjoint
in $L^2(\mathbb{M}, d\omega)$; see e.g.\ \cite[Theorem 5.2.3]{Da}.
The self-adjoint extension, which we denote by the same symbol, is non-negative.
The heat semigroup $\exp(t\Delta_{\mathbb{M}})$, $t \ge 0$, has an integral representation, cf.\ \cite[Chapter~VI]{Ch1},
$$
\exp(t\Delta_{\mathbb{M}})f(x) = \int_{\mathbb{M}} K_t^{\mathbb{M}}(x,y) f(y)\, d\omega(y),
	\qquad x \in \mathbb{M}, \quad t > 0, \quad f \in L^2(\mathbb{M},d\omega),
$$
where the kernel $K_t^{\mathbb{M}}(x,y)$ is strictly positive and smooth for
$(t,x,y) \in (0,\infty)\times \mathbb{M} \times \mathbb{M}$. See the general theory in e.g. \cite[Theorem 5.2.1]{Da}.
By means of an isometry of $\mathbb{M}$, we see that $K_t^{\mathbb{M}}(x,y)$ depends on $x$ and $y$ only through the distance
$d_{\mathbb{M}}(x,y)$.

The following sharp estimate for $K_t^{\mathbb{M}}$ is the most significant result of the paper.
\begin{thm} \label{thm:main}
Let $\mathbb{M} = \mathbb{M}^{d,\tilde{d}}$ be a compact rank-one Riemannian symmetric space of dimension $d\ge 1$
and antipodal dimension $\tilde{d}$. Then 
$$
K_t^{\mathbb{M}}(x,y)\simeq \big[t+\diam\mathbb{M}-d_{\mathbb{M}}(x,y)\big]^{-(d-\tilde{d}-1)/2}
	\frac{1}{t^{d/2}} \exp\bigg(-\frac{d_{\mathbb{M}}^{\,2}(x,y)}{4t}\bigg),
$$
uniformly in $x,y \in \mathbb{M}$ and $0 < t \le 1$.
The constants in the lower and upper estimates depend only on $d$ and $\tilde{d}$.
\end{thm}

For $t \ge 1$ one has $K_t^{\mathbb{M}}(x,y) \simeq 1$, see Section \ref{sec:desc}.
We mention that in the case of the Euclidean sphere, Molchanov \cite[Example 3.1]{M} obtained the bound of Theorem
\ref{thm:main} in the antipodal situation. There also a similar result
for other compact symmetric spaces is hinted.

As a corollary, we give a genuinely sharp estimate of the derivative of the heat kernel as a function of the geodesic distance.
It is a consequence of the proof of Theorem \ref{thm:main} given below, in particular \eqref{4.X}, and of
Lemma \ref{lem:diff}, Theorem \ref{thm:jac3} and \eqref{jac:long}.
Let $\mathfrak{K}_t^{\mathbb{M}}(\varphi)$, $\varphi \in [0,\diam\mathbb{M}]$, be the function defined by
$$
K_t^{\mathbb{M}}(x,y) = \mathfrak{K}_t^{\mathbb{M}}\big(d_{\mathbb{M}}(x,y)\big), \qquad x,y \in \mathbb{M}, \quad t > 0.
$$
\begin{cor} \label{cor:main}
Let $\mathbb{M} = \mathbb{M}^{d,\tilde{d}}$ be as in Theorem \ref{thm:main}. Then
$$
-\frac{\partial}{\partial \varphi} \mathfrak{K}_t^{\mathbb{M}}(\varphi) \simeq \varphi (\diam\mathbb{M}-\varphi)
	\begin{cases}
		(t+\diam\mathbb{M}-\varphi)^{-(d+1-\tilde{d})/2} {t^{-d/2-1}} \exp\left( - \frac{\varphi^2}{4t} \right), & \;\; \textrm{if}
			\;\; t\le 1, \\
		e^{-t(d-\tilde{d}/2)}, \;\; \textrm{if} \;\; t\ge 1,
	\end{cases}
$$
uniformly in $\varphi \in [0,\diam\mathbb{M}]$ and $t > 0$.
\end{cor}

\begin{proof}[{Proof of Theorem \ref{thm:main}}]
We first scale the Riemannian structure on $\mathbb{M}$ in such a way that all lengths are multiplied by the fixed factor
$\pi\slash \diam\mathbb{M}$. In particular, the diameter of $\mathbb{M}$ will now be $\pi$. After this scaling, some
objects will be marked by the symbol $\,\breve{}\,\,$, for instance the distance
$\breve{d}_{\mathbb{M}}(x,y) = (\pi\slash \diam\mathbb{M})\, d_{\mathbb{M}}(x,y)$.
The Laplace-Beltrami operator will then be replaced by
$\breve{\Delta}_{\mathbb{M}} =(\pi\slash \diam\mathbb{M})^{-2}\Delta_{\mathbb{M}}$.
In the heat equation and the heat kernel, the variable $t$ must be multiplied by $(\pi\slash \diam\mathbb{M})^2$,
and so the quantity $d^2_{\mathbb{M}}(x,y)/(4t)$ remains invariant.
The Riemannian volume measure $d\omega$ should also be scaled accordingly, but we prefer to scale it in such a way that we get
a probability measure $d\sigma$ on $\mathbb{M}$.

After this scaling, the heat kernel $\breve{K}_t^{\mathbb{M}}$ is defined by
$$
\exp(t\breve{\Delta}_{\mathbb{M}})f(x) = \int_{\mathbb{M}} \breve{K}_t^{\mathbb{M}}(x,y) f(y)\, d\sigma(y),
	\qquad x \in \mathbb{M}, \quad t > 0.
$$
It follows that
$$
\breve{K}_t^{\mathbb{M}}(x,y) = \omega(\mathbb{M})\, K^{\mathbb{M}}_{(\pi\slash \diam\mathbb{M})^{-2}t}\,(x,y).
$$
The reason for this scaling is that it connects the heat kernel in $\mathbb{M}$ with the Jacobi heat kernel
$G_t^{\a,\b}$ treated in the preceding section, as we now show.

Let first
$$
\a = \a\Big(\mathbb{M}^{d,\tilde{d}\,}\Big) = \frac{d}2-1, \qquad 
\b = \b\Big(\mathbb{M}^{d,\tilde{d}\,}\Big) = \frac{d-\tilde{d}}2-1.
$$
The spectrum of $-\breve{\Delta}_{\mathbb{M}}$ consists of the numbers
$\lambda_n^{\a,\b} = n(n+\a+\b+1)$, $n=0,1,2,\ldots$; see for instance \cite{Sch}.
The corresponding eigenspaces are finite dimensional.
More precisely, there exists an orthonormal basis of $L^2(\mathbb{M},d\sigma)$, namely
$S_{n,k}^{\mathbb{M}}$, $k=1, \ldots, \delta(n,\mathbb{M})$, $n=0,1,2,\ldots$,
consisting of the so-called ``spherical polynomials'' satisfying
$$
-\breve{\Delta}_{\mathbb{M}} S_{n,k}^{\mathbb{M}} = \lambda_n^{\a,\b} S_{n,k}^{\mathbb{M}}. 
$$
Here the dimensions $\delta(n,\mathbb{M})$ are known explicitly (see e.g.\ \cite{Sch}),
but this will not be needed for our purposes.
The self-adjoint operator $-\breve{\Delta}_{\mathbb{M}}$ can be characterized as
$$
-\breve{\Delta}_{\mathbb{M}}f = \sum_{n=0}^{\infty} \lambda_n^{\a,\b} \sum_{k=1}^{\delta(n,\mathbb{M})}
	\big\langle f, S_{n,k}^{\mathbb{M}} \big\rangle_{d\sigma} S_{n,k}^{\mathbb{M}},
$$
the domain consisting of those $f\in L^2(\mathbb{M},d\sigma)$ for which the series is $L^2$-convergent;
here $\langle \cdot,\cdot \rangle_{d\sigma}$ is the standard inner product in $L^2(\mathbb{M},d\sigma)$.
The heat kernel is now given by
$$
\breve{K}_t^{\mathbb{M}}(x,y) = \sum_{n=0}^{\infty} e^{-t\lambda_n^{\a,\b}} \sum_{k=1}^{\delta(n,\mathbb{M})}
	S_{n,k}^{\mathbb{M}}(x) \overline{S_{n,k}^{\mathbb{M}}(y)}, \qquad x,y \in \mathbb{M}, \quad t > 0.
$$

To find a more useful expression for $\breve{K}_t^{\mathbb{M}}$, we shall use the addition formula
\begin{equation} \label{add}
\sum_{k=1}^{\delta(n,\mathbb{M})} S_{n,k}^{\mathbb{M}}(x) \overline{S_{n,k}^{\mathbb{M}}(y)}
	= c_n^{\a,\b} P_n^{\a,\b}\big( \cos d_{\mathbb{M}}(x,y) \big), \qquad x,y \in \mathbb{M}, \quad n=0,1,2,\ldots,
\end{equation}
where
$$
c_n^{\a,\b} = \frac{\Gamma(\b+1)(2n+\a+\b+1)\Gamma(n+\a+\b+1)}{\Gamma(\a+\b+2)\Gamma(n+\b+1)},
$$
with the same understanding of the case $n=0=\a+\b+1$ as in \eqref{h_for}.
This is justified by a result of Gin\'e; see \cite{Gi,Ko} and also \cite{BKLT,BrDa}.
The Jacobi polynomials appear in this context since the so-called radial part of $\breve{\Delta}_{\mathbb{M}}$ is the
Jacobi differential operator, whose eigenfunctions are $P_n^{\a,\b} \circ \cos$.
Combining \eqref{add} with \eqref{h_for}, \eqref{P1_for} and \eqref{iden3}, one finds that
\begin{align} \nonumber
\breve{K}_t^{\mathbb{M}}(x,y) & = \sum_{n=0}^{\infty} c_n^{\a,\b} e^{-t\lambda_n^{\a,\b}}
	P_n^{\a,\b}\big(\cos \breve{d}_{\mathbb{M}}(x,y) \big) \\
& = \big\|P_0^{\a,\b}\big\|^2_{2,(\a,\b)} \sum_{n=0}^{\infty} e^{-t\lambda_n^{\a,\b}}
	\frac{P_n^{\a,\b}\big(\cos \breve{d}_{\mathbb{M}}(x,y) \big) P_n^{\a,\b}(1)}{h_n^{\a,\b}} \label{4.X} \\
& = \rho_{\a,\b}\big([-1,1]\big) \, G_t^{\a,\b}\big( \cos \breve{d}_{\mathbb{M}}(x,y), 1\big). \nonumber
\end{align}
Notice that this connection would have an even simpler form if we used the normalized probability measure in the Jacobi
context (however, we did not do so for the sake of compatibility with \cite{NoSj} and other papers).

Theorem \ref{thm:jac3} and \eqref{jac:long} now imply that for any $T>0$
$$
\breve{K}_t^{\mathbb{M}}(x,y) \simeq \big[t+\pi-\breve{d}_{\mathbb{M}}(x,y)\big]^{-(d-\tilde{d}-1)/2}
	\frac{1}{t^{d/2}} \exp\bigg(-\frac{{\breve{d}}^{\,2}_{\mathbb{M}}(x,y)}{4t}\bigg),
$$
uniformly in $x,y \in \mathbb{M}$ and $0 < t \le T$.
From this, Theorem \ref{thm:main} follows.
\end{proof}

\section{Heat kernel estimates on the unit ball} \label{sec:ball}

Let $\mu > -1/2$.
Denote by $|\cdot|$ the Euclidean norm in $\R^d$, $d \ge 2$, and let $B^d=\{x \in \R^d : |x| \le 1\}$
be the closed unit ball equipped with the measure $\W$ given by
\begin{align*}
\W (x) = \big(1-|x|^2\big)^{\mu - 1/2} \, dx.
\end{align*}
We consider the second-order differential operator
\begin{align*}
\eL f = - \Delta f 
+ 
\sum_{i,j=1}^d x_i x_j \frac{\partial^2 f}{ \partial x_i \, \partial x_j }
+
(2\mu + d) \sum_{i=1}^d x_i \frac{\partial f}{ \partial x_i },
\end{align*}
acting initially on polynomials in $B^d$. It is known that $\eL$ is symmetric, non-negative and essentially self-adjoint
in $L^2(\W)$, see e.g.\ \cite{DaiXu} or \cite{SjSz}. We denote by the same symbol the self-adjoint extension of $\eL$.

The associated heat semigroup $\exp(-t\eL)$ has an integral representation
\begin{equation*} 
\exp(-t\eL) f (x)
= \int_{B^d} h_t^\mu (x,y) f(y) \, \W(y), 
\qquad x \in B^d, \quad t > 0, \quad f \in L^2(\W),  
\end{equation*}
where the heat kernel $h_t^\mu (x,y)$ can be expressed as (see \cite[Corollary 11.1.8]{DaiXu} or \cite[(2.5)]{KPX2}) 
\begin{align*}
h_t^\mu (x,y)
=
\sum_{n=0}^{\8} e^{-tn( n + 2 \lambda_\mu )} P_n(W_{\mu}, x , y),
\qquad x, y \in B^d, \quad t>0.
\end{align*}
Here $\lambda_\mu = \mu + (d-1)/2$ and $P_n(W_{\mu}, x , y)$ is the reproducing kernel of the space of orthogonal
polynomials of degree $n$ with respect to $\W$. Under the restriction $\mu \ge 0$ this kernel can be expressed as 
\begin{align*}
P_n(W_{\mu}, x , y)
=
\frac{1}{W_{\mu} (B^d)} \frac{n + \lambda_\mu}{ \lambda_\mu }
\int C_n^{ \lambda_\mu } \Big(\langle x, y \rangle + u\sqrt{1-|x|^2} \sqrt{1-|y|^2} \Big)\, \pimu,
\end{align*}
$\langle \cdot,\cdot\rangle$ denoting the scalar product in $\mathbb{R}^d$.
See again \cite{DaiXu} or \cite{KPX2}, where the normalization is slightly different.
Here $d\Pi_{\mu - 1/2}$ is defined as in \eqref{Pimeas},
and $C_n^\lambda$ is the classical Gegenbauer (ultraspherical) polynomial of degree $n$ given by (cf.\ \cite{Sz})
\begin{align*}
C_n^\lambda (x) = \frac{\Gamma (2\lambda + n) \Gamma (\lambda + 1/2)}{ \Gamma (2\lambda) \Gamma (\lambda + n + 1/2) }
P_n^{\lambda - 1/2, \lambda - 1/2 } (x),
\qquad 
0 \neq \lambda > -1/2.
\end{align*}
Combining the above identities with \eqref{iden3} leads us to the following representation of $h_t^\mu (x,y)$,
see also \cite[(2.9)]{KPX2},
\begin{align} \label{iden1}
h_t^\mu (x,y) & =
\frac{2^{2\lambda_\mu} \Gamma (\lambda_\mu + 1/2)^2 }{ \Gamma (2\lambda_\mu + 1) W_{\mu} (B^d)}  \\ 
& \qquad \times
\int G_t^{ \lambda_\mu - 1/2, \lambda_\mu - 1/2} 
\Big( \langle x, y \rangle + u\sqrt{1-|x|^2} \sqrt{1-|y|^2}, 1 \Big) \, \pimu; \nonumber
\end{align}
here $x,y \in B^d$ and $t>0$.

The distance we use in this context is
\begin{align*}
\distB(x,y) = \arccos \Big( \langle x, y \rangle + \sqrt{1-|x|^2} \sqrt{1-|y|^2} \Big),
\qquad x,y \in B^d. 
\end{align*}
Observe that $\distB(x,y)$ coincides with the geodesic distance between the points $\big(x, \sqrt{1-|x|^2}\, \big)$ and
$\big(y, \sqrt{1-|y|^2}\, \big)$ on the unit sphere in $\R^{d+1}$. Thus $\diam B^d = \max_{x,y \in B^d} \distB(x,y) = \pi$.
Notice also that $x$ and $y$ are antipodes in the sense that $d_{B}(x,y)=\pi$ if and only if $|x|=1$ and $y=-x$.

\begin{thm} \label{thm:heatball}
Let $d\ge 2$ and $\mu \ge 0$ be fixed. Then 
\begin{align} \label{iden2}
h_t^\mu (x,y)
& \simeq
\big[t + \pi - \distB(x,y) \big]^{ - \lambda_\mu } 
\bigg( t + \frac{\sqrt{1 - |x|^2} \sqrt{1 - |y|^2} }{ \pi - \distB(x,y)  } \bigg)^{-\mu} 
\frac{1}{t^{d/2}}
\exp\bigg( - \frac{\distB^{\,2}(x,y)}{4t} \bigg),
\end{align}
uniformly in $x,y \in B^d$ and $0 < t \le 1$.
Here the fraction $\sqrt{1-|x|^2}\sqrt{1-|y|^2}/(\pi-d_B(x,y))$ is extended to a continuous function on
$B^d \times B^d$ with the value $0$ at antipodal points.
\end{thm}

We exclude $d=1$ in Theorem~\ref{thm:heatball}, since in this case our present setting is a special case of the Jacobi
context considered in Section \ref{sec:jac}. 
Note also that for all $\mu > -1/2$,
\begin{align*}
h_t^\mu (x,y) \simeq 1,
\qquad x,y \in B^d, \quad t \ge 1,
\end{align*}
which follows from the Gaussian bounds for $h_t^\mu (x,y)$ obtained in \cite{KPX1,KPX2,SjSz}.
When $\mu \ge 0$ this also follows from \eqref{iden1} and \eqref{jac:long}.

\begin{proof}[Proof of Theorem~\ref{thm:heatball}]
We first observe that the claimed continuity follows from the estimate
$$
\frac{\sqrt{1 - |x|^2} \sqrt{1 - |y|^2} }{\pi - \distB(x,y)} \simeq
\frac{\sqrt{1 - |x|^2} \sqrt{1 - |y|^2} }{(1+ \langle x,y\rangle + \sqrt{1 - |x|^2} \sqrt{1 - |y|^2})^{1/2}}
\le \big(1-|x|^2\big)^{1/4} \big(1-|y|^2\big)^{1/4}.
$$

Notice that Theorem \ref{thm:jac3} implies \eqref{jac:spec} with $\lambda = \lambda_\mu -1/2$.
Combining this with \eqref{iden1} we obtain
\begin{align*}
h_t^\mu (x,y) 
\simeq 
t^{-\lambda_\mu - 1/2}
\int \big( t + \pi - f(u,x,y) \big)^{-\lambda_\mu}
\exp\bigg( -\frac{f^2(u,x,y)}{4t} \bigg) \, \pimu,
\end{align*}
uniformly in $x,y \in B^d$ and $0 < t \le 1$;
here $f(u,x,y) = \arccos \big( \langle x, y \rangle + u\sqrt{1-|x|^2} \sqrt{1-|y|^2}\,\big)$.

We first show that the main contribution in the last integral comes from integrating over $[0,1]$.
This can be seen by reflecting $u \mapsto -u$ and using \eqref{estZ}
with $\gamma = - \lambda_{\mu}$, $\t = f(-u,x,y)$ and $\tau = f(u,x,y)$, for $u \in [0,1]$, $x,y \in B^d$.
Thus,
\begin{align*}
h_t^\mu (x,y) 
\simeq 
t^{-\lambda_\mu - 1/2}
\int_{[0,1]} \big( t + \pi - f(u,x,y) \big)^{-\lambda_\mu}
\exp\bigg( -\frac{f^2(u,x,y)}{4t} \bigg) \, \pimu,
\end{align*}
uniformly in $x,y \in B^d$ and $0 < t \le 1$.
Now applying Lemma~\ref{lem:FVII} (with $\nu = \mu-1/2$, $\xi = \lambda_\mu$,
$A = \langle x, y \rangle$, $B = \sqrt{1-|x|^2} \sqrt{1-|y|^2}$ and $D = 4t$)
and using the fact that $\lambda_\mu = \mu + (d-1)/2$, we get \eqref{iden2}.
\end{proof}

\section{Heat kernel estimates on the simplex} \label{sec:simplex}

Let $d \ge 2$, and let $\kappa = (\kappa_1, \ldots, \kappa_{d+1}) \in [0,\infty)^{d+1}$ be a multi-parameter.
We will write $|\kappa|$ for the length of $\kappa$ (sum of coordinates).
Denote by $\mathbb{V}^d$ the unit simplex in $\R^d$,
\begin{align*}
\mathbb{V}^d = \bigg\{ x \in \R^d : x_j \ge 0,\; j =1, \ldots, d,  \textrm{\, and \, } \sum_{j=1}^d x_j \le 1 \bigg\}.
\end{align*} 
Equip $\mathbb{V}^d$ with the measure
\begin{align*}
dU_{\kappa} (x) = \prod_{j=1}^{d+1} x_j^{\kappa_j - 1/2} \, dx_1\ldots dx_d;
\end{align*}
here and later on we denote $x_{d+1} = 1-|x|_1$,
where $|x|_1 = x_1+\ldots + x_d$. 

Consider the second-order differential operator
\begin{align*}
\mathbb{L}_{\kappa} f 
= - \sum_{j=1}^d x_j \frac{\partial^2 f}{ \partial x_j^2 }  
+ 
\sum_{i,j=1}^d x_i x_j \frac{\partial^2 f}{ \partial x_i \, \partial x_j }
-
\sum_{j=1}^d \Big( \kappa_j + 1/2 - \big( |\kappa| + (d+1)/2 \big) x_j \Big)  \frac{\partial f}{ \partial x_j },
\end{align*}
acting initially on polynomials in $\mathbb{V}^d$. It is known that $\mathbb{L}_{\kappa}$ is symmetric, non-negative and essentially
self-adjoint in $L^2(dU_{\kappa})$; see \cite[Proposition 3.1]{KPX2}. We denote by the same symbol the self-adjoint extension 
of $\mathbb{L}_{\kappa}$.

The associated heat semigroup $\exp(-t\mathbb{L}_{\kappa})$ has an integral representation
\begin{equation*} 
\exp(-t \mathbb{L}_{\kappa}) f (x)
= \int_{\mathbb{V}^d} H_t^{\kappa} (x,y) f(y) \, dU_{\kappa} (y), 
\qquad x \in \mathbb{V}^d, \quad t > 0, 
\end{equation*}
where the heat kernel 
is given by
\begin{align*}
H_t^{\kappa} (x,y)
=
\sum_{n=0}^{\8} e^{-tn( n + \lambda_\kappa )} P_n(U_{\kappa} , x , y),
\qquad x, y \in \mathbb{V}^d, \quad t>0.
\end{align*}
Here $\lambda_\kappa = |\kappa|+ (d-1)/2$, and $P_n( U_{\kappa}, x , y)$ is the reproducing kernel of the space of orthogonal
polynomials of degree $n$ with respect to $dU_{\kappa}$. By means of \cite[Corollary 13.1.6]{DaiXu}
(note that our normalization differs from that in \cite{DaiXu}), we get
\begin{align*}
P_n(U_{\kappa} , x , y)
=
\frac{p_n^{\lambda_{\kappa} - 1/2, -1/2} (1)}{U_{\kappa} (\mathbb{V}^d)}
\int_{[-1,1]^{d+1}} p_n^{\lambda_{\kappa} - 1/2, -1/2} \big( 2 z^2(u,x,y) - 1 \big) \, \piku
\end{align*}
for $x, y \in \mathbb{V}^d$ and $n \in \N$. 
Here $\mathbf{1}=(1,\ldots,1) \in \mathbb{R}^{d+1}$, and $d\Pi_{\kappa - \mathbf{1}/2}$ is a tensor product of
one-dimensional measures defined in \eqref{Pimeas}. Further, $p_n^{\a,\b}$ are re-normalized Jacobi polynomials,
\begin{align*}
p_n^{\a,\b} (x) = \bigg( 
\frac{2^{\a + \b + 1}\Gamma(\a+1) \Gamma(\b+1) }{h_n^{\a,\b} \Gamma(\a+\b+2) } \bigg)^{1/2} 
P_n^{\a,\b} (x)
\end{align*}
and
\begin{align*}
z(u,x,y) = \sum_{j=1}^{d+1} u_j\sqrt{x_j y_j}.
\end{align*}
Combining the above formulas with \eqref{iden3}, we get, see also \cite[(3.7)]{KPX2},
\begin{align*} 
H_t^{\kappa} (x,y) 
=
\frac{\sqrt{\pi}\,2^{\lambda_\kappa} \Gamma (\lambda_\kappa + 1/2)}
{ \prod_{j=1}^{d+1} \Gamma (\kappa_j + 1/2) }
\int_{[-1,1]^{d+1}} G_t^{ \lambda_\kappa - 1/2, - 1/2} 
\big( 2 z^2(u,x,y) - 1, 1 \big)\, \piku,
\end{align*}
for $x,y \in \mathbb{V}^d$ and $t>0$.
By means of the first identity of Lemma \ref{lem:qiden},
we arrive from here at the following nice expression for the simplex heat kernel:
\begin{equation} \label{fhs}
H_t^\kappa (x,y)
= C_{d,\kappa} 
\int_{[-1,1]^{d+1}} G_{t/4}^{ \lambda_\kappa - 1/2, \lambda_\kappa - 1/2} 
\big( z(u,x,y) , 1 \big) \, \piku, \qquad x,y \in \mathbb{V}^d, \quad t > 0.
\end{equation}

To state our bounds for $H_t^\kappa$, we recall
the relevant distance on the simplex, cf.\ \cite{DaiXu}, 
\begin{align*}
d_{\mathbb{V}} (x,y) = \arccos \big( z(\mathbf{1},x,y)\big),
\qquad x,y \in \mathbb{V}^d.
\end{align*}
As easily verified, $d_{\mathbb{V}}$ can be obtained from the geodesic distance on the unit sphere
via the bijection $(x_j)_{j=1}^{d}\mapsto (\sqrt{x_j})_{j=1}^d$
from $\mathbb{V}^d$ to the subset $\{x \in B^d : x_j \ge 0, \, j=1, \ldots, d\}$ of $B^d$.

\begin{thm} \label{thm:heatsim}
Let $d\ge 2$ and let $\kappa \in [0,\infty)^{d+1}$ be fixed. 
Then
\begin{equation} \label{iden200} 
H_t^\kappa (x,y)
\simeq
\bigg( \prod_{j=1}^d \big( t + \sqrt{x_j y_j}\big)^{-\kappa_j}\bigg) \Big( t + \sqrt{(1-|x|_1)(1-|y|_1)}\Big)^{-\kappa_{d+1}}
	\frac{1}{t^{d/2}} \exp\bigg( -\frac{d_{\mathbb{V}}^{\,2}(x,y)}t\bigg),
\end{equation} 
uniformly in $x,y \in \mathbb{V}^d$ and $0 < t \le 1$.
\end{thm}

We remark that in the special case $\kappa = (0,\ldots,0,\mu)$ the measure $dU_{\kappa}$ corresponds to $2^d \W$
under the bijection described before Theorem \ref{thm:heatsim}, and $\mathbb{L}_{\kappa}$ corresponds to $\eL/4$.
This is why the denominator in the exponential factor in \eqref{iden200} is $t$ and not $4t$.

For large $t$ and all $\kappa \in (-1/2,\infty)^{d+1}$, one has
$$
H_t^\kappa (x,y) \simeq 1, \qquad x,y \in \mathbb{V}^d, \quad t \ge 1,
$$
which follows from the Gaussian bounds for $H_t^\kappa (x,y)$ obtained in \cite{KPX1,KPX2}.
When $\kappa \in [0,\infty)^{d+1}$, this can also easily be deduced from \eqref{fhs} and \eqref{jac:long}.

\begin{proof}[Proof of Theorem~\ref{thm:heatsim}]
Plugging into \eqref{fhs} the estimate \eqref{jac:spec} with $\lambda=\lambda_{\kappa}-1/2=|\kappa|+d/2-1$, we get
$$
H_t^{\kappa}(x,y) \simeq t^{-\lambda_{\kappa}-1/2} \int_{[-1,1]^{d+1}} \big[ t+\pi -f(u,x,y)\big]^{-\lambda_{\kappa}}
	\exp\bigg( -\frac{f^2(u,x,y)}{t} \bigg)\, d\Pi_{\kappa-\mathbf{1}/2}(u),
$$
uniformly in $x,y \in \mathbb{V}^d$ and $0 < t \le 1$; here $f(u,x,y) = \arccos ( z(u,x,y) )$.

We claim that the main contribution to the last integral comes from integrating over the subcube $[0,1]^{d+1}$.
Indeed, let $G$ be the finite reflection group generated by reflections of $\mathbb{R}^{d+1}$ in the hyperplanes
$u_j = 0$, $j=1, \ldots, d+1$.
The claim follows from the $G$-invariance of the measure $d\Pi_{\kappa-\mathbf{1}/2}$
and of the set of integration, and from the bound \eqref{estZ} specified to 
$\gamma = -\lambda_{\kappa}$, $\t = f(\sigma u,x,y)$ and $\tau = f(u,x,y)$ 
with $\sigma \in G$, $u \in [0,1]^{d+1}$, $x,y \in \mathbb{V}^d$ and $t>0$;
notice that $f(u,x,y) \in [0,\pi/2]$ and $0 \le f(u,x,y) \le f(\sigma u,x,y) \le \pi$ when $u \in [0,1]^{d+1}$.

We then have
$$
H^{\kappa}_t(x,y) \simeq t^{-\lambda_{\kappa}-1/2} \int_{[0,1]^{d+1}}
	\exp\bigg( -\frac{f^2(u,x,y)}{t} \bigg)\, d\Pi_{\kappa-\mathbf{1}/2}(u),
$$
uniformly in $x,y \in \mathbb{V}^d$ and $0 < t \le 1$. Now we iterate the above integral and apply Lemma~\ref{lem:FVII}
$d+1$ times. Integrating first with respect to $d\Pi_{\kappa_1-1/2}(u_1)$, we apply Lemma \ref{lem:FVII} with
$\nu = \kappa_1-1/2$, $\xi=0$, $A=\sum_{j=2}^{d+1} u_j \sqrt{x_j y_j}$, $B = \sqrt{x_1 y_1}$ and $D=t$, getting
\begin{align*}
H_t^{\kappa}(x,y) & \simeq t^{-\lambda_{\kappa}-1/2}t^{\kappa_1} \big( t+ \sqrt{x_1 y_1}\big)^{-\kappa_1} \\ & \qquad \times
	\int_{[0,1]^d} \exp\Bigg( -\frac{f^2\big((1,u_2,\ldots,u_{d+1}),x,y\big)}{t} \Bigg)\, d\Pi_{\kappa_2-1/2}(u_2)\ldots
		d\Pi_{\kappa_{d+1}-1/2}(u_{d+1}).
\end{align*}
Repeating this step with the remaining $d$ integrals and applying each time Lemma \ref{lem:FVII} with suitably chosen
parameters, one arrives exactly at the desired estimate \eqref{iden200}.
\end{proof}


\end{document}